\documentclass[11pt]{amsart}
\usepackage{tikz}
\usetikzlibrary{decorations.markings}
\usetikzlibrary{shapes}
\usetikzlibrary{backgrounds}
\usepackage{subfig}
\usepackage{amssymb} 
\usepackage[all,cmtip]{xy}

\usepackage{hyperref}
\def\thetitle{{Anti-trees and right-angled Artin subgroups of braid groups}}
\hypersetup{
    pdftitle=   \thetitle,
   pdfauthor=  {Sang-hyun Kim and Thomas Koberda}
}
\usepackage{enumerate}
\makeatletter
\let\@@enum@org\@@enum@
\def\@@enum@[#1]{\@@enum@org[\normalfont #1]}
\makeatother

\newtheorem{thm}{Theorem}
\newtheorem{lem}[thm]{Lemma}
\newtheorem{cor}[thm]{Corollary}

\theoremstyle{remark}

\newtheorem*{rem}{Remark}

\theoremstyle{definition}
\newtheorem{defn}[thm]{Definition}

\newcommand\co{\colon}
\newcommand\bZ{\mathbb{Z}}

\newcommand\Aut{\operatorname{Aut}}

\newcommand\PMod{\operatorname{PMod}}

\newcommand\supp{\operatorname{supp}}

\newcommand\lk{\operatorname{Lk}}
\newcommand\st{\operatorname{St}}

\newcommand\Gam{\Gamma}

\newcommand\Mod{\operatorname{Mod}}
\newcommand\Id{\operatorname{Id}}

\newcommand\cP{\mathcal{P}}

\newcommand\cG{\mathcal{G}}

\newcommand\yt{\widetilde}
\newcommand\diff{\ensuremath{{\mathrm{Symp}(D^2,\partial D^2)}}}
\newcommand\symps{\ensuremath{{\mathrm{Symp}(S^2)}}}

\begin{document}

\title\thetitle

\author{Sang-hyun Kim}
\address{Department of Mathematical Sciences, Seoul National Univeristy, Seoul 151-747, Republic of Korea}
\email{s.kim@snu.ac.kr}

\author{Thomas Koberda}
\address{Department of Mathematics, Yale University, 20 Hillhouse Ave, New Haven, CT 06520, USA}
\email{thomas.koberda@gmail.com}
\date{\today}
\begin{abstract}
We prove that an arbitrary right-angled Artin group $G$ admits a quasi-isometric group embedding into a right-angled Artin group defined by the opposite graph of a tree,
and consequently, into a pure braid group.
It follows that $G$ is a quasi-isometrically embedded subgroup of the area-preserving diffeomorphism groups of the 2--disk and of the 2--sphere with $L^p$-metrics for suitable $p$.
Another corollary is that, there exists a closed hyperbolic manifold group of each dimension which admits a quasi-isometric group embedding into a pure braid group.
Finally, we show that the isomorphism problem, conjugacy problem, and membership problem are unsolvable in the class of finitely presented subgroups of braid groups.
\end{abstract}

\maketitle
\section{Introduction}
A \emph{right-angled Artin group} on a finite graph $\Gamma$ is a group generated by the vertices of $\Gamma$ with commutation relations defined by the adjacency relations in $\Gamma$.  In this article, we adopt the opposite of the usual convention used for discussing right-angled Artin groups, and we write \[{G(\Gamma)}\cong\langle V(\Gamma)\mid [v_i,v_j]=1 \textrm{ if and only if } \{v_i,v_j\}\notin E(\Gamma)\rangle.\]  Here, $V(\Gamma)$ and $E(\Gamma)$ denote the set of vertices and edges of $\Gamma$.
\subsection{Main results}
For two groups $G$ and $H$ equipped with metrics, a \emph{quasi-isometric group embedding from $G$ to $H$} (with quasi-isometry constant $C$) is an injective group homomorphism $f\co G\to H$  such that the following estimate holds for every $x$ and $y$ in $G$: \[d_G(x,y)/C-C\le d_H(f(x),f(y))\le C d_G(x,y) + C.\] 
We will also say that $G$ is a \emph{quasi-isometrically embedded subgroup} of $H$. 
We will endow a word-metric to each finitely generated group. Note that every quasi-isometric group embedding between finitely generated groups is bi-Lipschitz.

The principal result of this article is the following:

\begin{thm}\label{t:main}
For each finite graph $\Gamma$, there exists a finite tree $T$ such that ${G(\Gamma)}$ admits a quasi-isometric group embedding into $G(T)$.
\end{thm}

In the usual convention for describing right-angled Artin groups, the tree $T$ should be replaced with its opposite graph $T^{opp}$, a graph which could reasonably be called an \emph{anti-tree}.  
The tree $T$ in Theorem \ref{t:main} is not produced in an ad hoc manner, but is rather built as a subtree of the universal cover of $\Gamma$.

Let us summarize some consequences of the main theorem. The reader is referred to the last section for more details and background.

We write $B_n$ for the \emph{braid group} on $n$--strands, which is identified with the mapping class group of the $n$--punctured disk $D_n$ fixing the boundary pointwise.
The \emph{pure braid group} $P_n<B_n$ is the kernel of the natural puncture permutation representation \[B_n\to S_n<\Aut(H_1(D_n,\bZ)).\]

We let $\diff$ be the group of area-preserving diffeomorphisms (\emph{symplectomorphisms}) of the unit 2--disk that are the identity in a neighborhood of the boundary. We endow $\diff$ with the $L^p$--metric $d_p$; see Section~\ref{ss:sph} and~\ref{ss:disk}.
Crisp and Wiest \cite{CW2007} proved that the right-angled Artin group on each planar graph (more generally, each \emph{planar type} graph) admits quasi-isometric group embeddings into $P_n$ for some $n$, and into $(\diff,d_2)$.
They asked whether or not an arbitrary right-angled Artin group embeds into $P_n$ for some $n$~\cite{CW2004}.

We denote by $\symps$ the area-preserving diffeomorphism group of the unit 2--sphere $S^2$, again equipped with the $L^p$--metric. Note that $\symps$ is equal to the group of Hamiltonian symplectomorphisms on $S^2$~\cite[1.4.H]{Polterovich2001}. M.~Kapovich showed that every right-angled Artin group embeds into $\symps$, asking whether or not this embedding can be chosen to be quasi-isometric with respect to $L^2$-metric~\cite{Kapovich2012}.
We generalize the result of Crisp and Wiest, and give an affirmative answer to the $L^p$-version of Kapovich's question for $p>2$.

\begin{cor}\label{c:braid}
Every right-angled Artin group admits quasi-isometric group embeddings into the following groups:
\begin{enumerate}
\item
$P_n$ for some $n$,
\item
$\diff$ with $L^p$-metric for $1\le p\le\infty$,
\item
$\symps$ with $L^p$-metric for $2<p\le \infty$.
\end{enumerate}
\end{cor}

Although largely influenced by their papers (\cite{CW2007}, ~\cite{Kapovich2012}), our proof does not rely on the above mentioned results of Crisp--Wiest and Kapovich.
In Corollary~\ref{c:braid} part (1), our analysis tells us that the number of strands $n$ can be chosen so that $\log_2 \log_2 n\le m^2$ where $m$ is the number of the generators of the right-angled Artin group. 

Corollary \ref{c:braid} part (1) can be applied to recover an algebraic result of Baudisch concerning the structure of two--generated subgroups of right-angled Artin groups:

\begin{cor}[See \cite{Baudisch1981}]\label{c:baud}
Every two--generated subgroup of a right-angled Artin group is either abelian or free.
\end{cor}

A group $G$ is called \emph{special} if $G = \pi_1(X)$ for some compact CAT(0) cube complex $X$ that locally isometrically embeds into the Salvetti complex of a right-angled Artin group. If a group has a finite-index special subgroup, it is called \emph{virtually special}~\cite{HW2008}.
The class of virtually special groups is surprisingly large, including all the finite-volume hyperbolic 3--manifold groups~\cite{HW2008,Wise2011,Agol2013}.

\begin{cor}\label{c:vsp}
A special group admits a quasi-isometric group embedding into a pure braid group. 
\end{cor}

In particular, every finite-volume hyperbolic 3--manifold group is virtually a quasi-isometrically embedded subgroup of a pure braid group. 
Theorem \ref{t:main} also furnishes examples of quasi-isometrically embedded higher dimensional closed hyperbolic manifold subgroups of a pure braid group, which to the authors' knowledge are the first such examples.

\begin{cor}\label{c:hypn}
For each $n\geq 2$, there exists an $n$--dimensional closed hyperbolic manifold whose fundamental group admits a quasi-isometric group embedding into a pure braid group.
\end{cor}
In particular, we will see that the fundamental group of a manifold cover of the 4--dimensional all-right hyperbolic 120--cell orbifold quasi-isometrically embeds into a pure braid group. This answers a question by Crisp and Wiest in \cite{CW2007}.

Recall that the \emph{isomorphism problem} asks if, given a class $\mathfrak{C}$ of finitely presented groups, whether there exists an algorithm which on an input of two members $A,B\in\mathfrak{C}$, halts and determines whether or not $A\cong B$.  In \cite{Bridson2012}, Bridson uses right-angled Artin subgroups of mapping class groups to show that the isomorphism problem is unsolvable when $\mathfrak{C}$ is the class of finitely presented subgroups of a sufficiently high genus surface mapping class group.  Theorem \ref{t:main} allows us to replace the surface mapping class group with a braid group:

\begin{cor}\label{c:isoprob}
Let $\mathfrak{C}_n$ be class of finitely presented subgroups of the planar braid group $B_n$.  If $n$ is sufficiently large then the isomorphism problem for $\mathfrak{C}_n$ is unsolvable.
\end{cor}

Recall that the \emph{conjugacy problem} in a finitely presented group $H$ asks whether there exists an algorithm which on an input of two elements $g,h\in H$, halts and determines whether or not $g$ and $h$ are conjugate to each other in $H$.  The \emph{membership problem} for a subgroup $K<H$ asks whether there exists an algorithm which on an input of a word $w$ in the generators of $H$, halts and determines whether or not the element of $H$ represented by $w$ lies in $K$.  With this setup, Theorem \ref{t:main} has the following consequence for braid groups:

\begin{cor}\label{c:membership}
For all $n$ sufficiently large, there exists a finitely presented subgroup $H<B_n$ such that the conjugacy problem for $H$ is unsolvable and for which the membership problem is unsolvable.
\end{cor}

The word problem is solvable among all finitely presented subgroups of braid groups by virtue of their residual finiteness.

\subsection{Notes and references}
In \cite{Kapovich2012}, M. Kapovich shows that every right-angled Artin group ${G(\Gamma)}$ embeds in the group of hamiltonian symplectomorphisms of the 2--sphere.   In the second part of his proof, Kapovich resolves technical difficulties arising from right-angled Artin groups on non--planar graphs.
As trees are planar, Theorem~\ref{t:main} simplifies Kapovich's argument. 
We will also strengthen Kapovich's theorem to Corollary~\ref{c:braid} part (3).

Right-angled Artin subgroups of right-angled Artin groups and of mapping class groups have been studied by 
various authors from several points
 of view.  The reader may consult \cite{Kim2008}, \cite{CLM2012}, \cite{CP2001}, \cite{CW2007}, \cite{KK2013}, \cite{KK2014b}, \cite{KK2013b} and \cite{Koberda2012}, for instance.

Mapping class groups have proven to be a fruitful setting for the study of right-angled Artin groups for their own sake.  The setup which most easily lends itself to analysis is when one has a homomorphism \[\phi:{G(\Gamma)}\to\Mod(S)\] such that $\phi$ maps each vertex generator of ${G(\Gamma)}$ to a mapping class with a connected support~\cite{CLM2012,Koberda2012}.
In the absence of the connectedness of the support of a mapping class, various algebraic pathologies may arise (see \cite{Koberda2012}, for instance).

In the right-angled Artin group setting, the analogue of a mapping class with disconnected support is an element of a right-angled Artin group that is not a \emph{pure factor}, an element $1\neq g\in{G(\Gamma)}$ which can be written as a product $g=g_1\cdot g_2$ of nontrivial elements with no common powers such that $[g_1,g_2]=1$.
If \[\phi:G(\Lambda)\to{G(\Gamma)}\] is a homomorphism of right-angled Artin groups where vertex generators of $G(\Lambda)$ are sent to elements of ${G(\Gamma)}$ which are not pure factors, the same algebraic pathologies which occur in the mapping class group setting can again occur.

Oftentimes, various fortuitous circumstances allow one to circumvent the algebraic pathologies which occur when studying homomorphisms from a right-angled Artin group to another right-angled Artin group, or to a mapping class group.  Such analyses have been carried out in \cite{CDK2013} and \cite{KK2013c} in the right-angled Artin group and mapping class group cases respectively, for instance.

This article is entirely concerned with injective maps ${G(\Gamma)}\to G(T)$ where vertex generators of ${G(\Gamma)}$ are sent to elements of $G(T)$ which are never pure factors, and with compositions of such maps with injective homomorphisms $G(T)\to P_n$, so that the vertex generators of ${G(\Gamma)}$ are always sent to mapping classes with disconnected support.  

Let us comment on the precursors of Theorem~\ref{t:main}. 
Crisp and Wiest proved that ${G(\Gamma)}$ embeds into $G(\Lambda)$ if $\Lambda$ is a finite cover of $\Gamma$~\cite[Proposition 19]{CW2004},
and M. Kapovich extended this result to the case when $\Lambda$ is 
a finite \emph{orbi-cover} of $\Gamma$~\cite[Lemma 2.3]{Kapovich2012}.
The methods of this paper was also inspired by the analysis in \cite{CDK2013} of an injective map $G(C_5)\to G(L_8)$, where $C_5$ denotes the cycle on five vertices and where $L_8$ denotes a path on eight vertices. 

\section{Acknowledgements}
The authors thank M.~Brandenbursky, C.~Cho, S.~Lee and Y.~Minsky for helpful conversations. 
The second named author was supported by National Research Foundation of Korean Government (NRF-2013R1A1A1058646) and Samsung Science and Technology Foundation (SSTF-BA1301-06).
The second named author is partially supported by NSF grant DMS-1203964.  The authors thank the referee for several helpful comments.

\section{Preliminaries}
Every graph in this paper will be assumed to be simplicial. 
For two graphs $X$ and $Y$, we write $X\le Y$ if a graph $X$ is an \emph{induced} subgraph of another graph $Y$. This means, $X$ is a subgraph of $Y$ such that \[E(X) = \binom{V(X)}{2}\cap E(Y).\] 
If $S$ is a subset of $V(Y)$, we often identify $S$ with the induced subgraph of $Y$ on $S$. 
For a vertex $x$ in a graph $X$, the \emph{link} and the \emph{star} of $v$ are defined respectively as:
\[
\lk(v) = \{u\in V(X)\co \{u,v\}\in E(X)\}, \st(v) = \lk(v)\cup \{v\}.\]

Let $\Gamma$ be a finite graph.
Each element in $V(\Gamma)\cup V(\Gamma)^{-1}$ is called a \emph{letter} in ${G(\Gamma)}$.
A \emph{word} in ${G(\Gamma)}$ is a finite sequence of letters, and usually written as a multiplication of letters.
The \emph{word length} of an element $g\in {G(\Gamma)}$ is the length of a shortest word representing $g$
and denoted as $\|g\|$.
A word $w$ is \emph{reduced} if its length realizes the word length.
The \emph{support} of $g\in {G(\Gamma)}$ is the set of the vertices $v$ such that $v$ or $v^{-1}$ appears in a reduced word for $g$.
The support of $g$ is denoted as $\supp(g)$.
For example, $\supp(a^{-1}bb^{-1})=\{a\}$ if $a$ and $b$ are vertices.

Let $w$ be a (possibly non-reduced) word in ${G(\Gamma)}$ and $v$ be a vertex of $\Gamma$.
A \emph{cancellation of $v$ in $w$} is a subword \[v^{\pm1}w' v^{\mp1}\] of $w$
such that $\supp(w')\cap \lk(v)=\varnothing$. If, furthermore, no letters in $w'$ are equal to $v$ or $v^{-1}$ then we say 
the word \[v^{\pm1}w' v^{\mp1}\] is an \emph{innermost cancellation of $v$ in $w$}.
If $v$ or $v^{-1}$ is a letter in $w$ and $v\not\in\supp w$, then there is a cancellation of $v$ in $w$ and hence, there is an innermost one; this follows from a well-known solution to the word problem in right-angled Artin groups~\cite{Wrathall1988,Charney2007}.
Hence every non-reduced word contains an innermost cancellation.
If there are no cancellations of a vertex $v$ in $w$,
then $v$ ``survives'' in each reduced word $w'$ representing $w$,
in the sense that the number of occurrences of $v$ or $v^{-1}$ in $w$ is the same as 
that in $w'$.

\section{Proof of Theorem~\ref{t:main}}
For the purposes of Theorem \ref{t:main}, we may assume that $\Gamma$ is connected.  For otherwise we can write $\Gamma=\Gamma_1\coprod\Gamma_2$ so that ${G(\Gamma)}=G(\Gamma_1)\times G(\Gamma_2)$.  
If there are trees $T_1,T_2$ and a quasi-isometric group embedding from $G(\Gamma_i)$ into $G(T_i)$ for each $i=1,2$,
then we simply let $T$ be the tree obtained by joining a vertex in $T_1$ to another vertex in $T_2$ by a length--two path.
Since there is a natural isometric embedding from $G(T_1)\times G(T_2)$ into $G(T)$, we have a quasi-isometric group embedding from ${G(\Gamma)}$ into $G(T)$.

From now on, we will let $\Gamma$ be a finite, connected graph and $p\co \yt \Gamma\to \Gamma$ be its universal cover.
We will fix an arbitrary order on $V(\yt\Gamma)$.
Every induced subgraph of $\yt\Gamma$ is a forest.
Let $T$ be a finite induced subgraph of $\yt{\Gamma}$. Then $T$ induces a group homomorphism $\phi(\Gamma,T):{G(\Gamma)}\to G(T)$ defined by \[\phi(\Gamma,T):v\mapsto\prod_{t\in p^{-1}(v)\cap T}t,\] where 
we define this product to be the identity if the indexing set is empty.  
Since no two vertices in $p^{-1}(v)$ are adjacent, the above product is well-defined.

Suppose $w$ is a word in ${G(\Gamma)}$ written as $x_1^{e_1} x_2^{e_2} \cdots x_\ell^{e_\ell}$ for $x_1,x_2,\ldots,x_\ell\in V(\Gamma)$
and $e_1,e_2,\ldots,e_\ell=\pm1$.
Let us write the product \[\phi(\Gamma,T)(x_i) = \prod_{t\in p^{-1}(x_i)\cap T}t\] in the increasing order.
We define
the \emph{$\phi(\Gamma,T)$--homomorphic word of $w$} as the word
\[\prod_{i=1}^\ell (\phi(\Gamma,T)(x_i))^{e_i}.\]
Note that the above (possibly non-reduced) word represents $\phi(\Gamma,T)(w)$ in $G(T)$.

\begin{defn}\label{d:surviving}
Let $T$ be a finite induced subgraph of $\yt\Gamma$ and $F\subseteq V(T)$.
We say $\phi(\Gamma,T)$ is \emph{$F$-surviving} if for every reduced word $w$ in ${G(\Gamma)}$ and $v\in F$,
the $\phi(\Gamma,T)$--homomorphic word of $w$ does not have a cancellation of $v$.
\end{defn}

We make some immediate observations:
\begin{lem}\label{l:surv}
Let $T'$ be a finite induced forest in $\yt\Gamma$ and $T\le T'$.
\begin{enumerate}
\item
For each $w\in {G(\Gamma)}$, we have $\supp(\phi(\Gamma,T)(w))\subseteq\supp(\phi(\Gamma,T')(w))$.
In particular, $\ker \phi(\Gamma,T')\le \ker \phi(\Gamma,T)$.
\item
Suppose $F'\subseteq F\subseteq V(T)$.
If $\phi(\Gamma,T)$ is $F$-surviving, then $\phi(\Gamma,T')$ is $F'$-surviving.
\item
For $F_1,F_2\subseteq V(T)$,
if $\phi(\Gamma,T)$ is $F_1$-surviving and $F_2$-surviving, then 
$\phi(\Gamma,T)$ is $(F_1\cup F_2)$-surviving.
\end{enumerate}
\end{lem}

\begin{proof}
There is a natural quotient map $\eta\co G(T')\to G(T)$ sending the vertices in $T'\setminus T$ to the identity.
Since $\supp\eta(w)\subseteq\supp(w)$ for $w\in G(T')$ and $\phi(\Gamma,T) = \eta\circ\phi(\Gamma,T')$,
we obtain (1). The parts (2) and (3) are clear from definitions.
\end{proof}

\begin{lem}\label{l:surv-qi}
Let $T$ be a finite subtree of $\yt\Gamma$ and $F\subseteq V(T)$ such that $p(F) = V(\Gamma)$.
If $\phi(\Gamma,T)$ is $F$-surviving then $\phi(\Gamma,T)$ is a quasi-isometric group embedding.
\end{lem}

\begin{proof}
Let $w$ be a reduced word in ${G(\Gamma)}$.
Note that $\| \phi(\Gamma,T)(w)\| \le | V(T) | \| w \|$. 
On the other hand, the $F$-surviving condition implies that some reduced word for $\phi(\Gamma,T)(w)$ contains the $\phi(\Gamma,F)$--homomorphic word of $w$ as a subsequence. This gives $\|\phi(\Gamma,T)(w)\| \ge \|w\|$ and in particular, $\phi(\Gamma,T)$ is injective.
\end{proof}

If $v$ is a vertex of $\Gamma$, we denote by $\Gamma\setminus v$ the induced subgraph of $\Gamma$ on $V(\Gamma)\setminus\{v\}$.
The following lemma is a key inductive step in the proof of Theorem~\ref{t:tree}.

\begin{lem}\label{l:key}
Let $v$ be a vertex of $\Gamma$
and $T_0$ be a finite tree in $\yt\Gamma$.
If the restriction of $\phi(\Gamma,T_0)$ to $G(\Gamma\setminus v)$ is injective,
then for each vertex $v'\in p^{-1}(v)$ there exists a finite tree $T$ in $\yt\Gamma$ containing $T_0\cup\{v'\}$
such that $\phi(\Gamma,T)$ is $v'$-surviving
and injective.
\end{lem}

\begin{proof}
Choose a finite tree $T_1\le \yt\Gamma$ containing $T_0\cup\st(v')$.
Let $\Sigma$ be the set of deck transformations $\sigma\co \yt\Gamma\to \yt\Gamma$ such that 
$\sigma(T_1)\cap T_1\ne\varnothing$. This set $\Sigma$ is finite since the deck transformation group of $\yt{\Gamma}$ over $\Gamma$ acts freely and simplicially.
Define \[T = \bigcup_{\sigma\in\Sigma} \sigma(T_1).\]

We first claim that $\phi(\Gamma,T)$ is $v'$-surviving. 
For this, let us suppose $w$ is a nontrivial reduced word in ${G(\Gamma)}$ such that the $\phi(\Gamma,T)$--homomorphic word of $w$ has an innermost cancellation of $v'$. We then find a subword \[v^{\pm1} w_1 v^{\mp1}\] of $w$ such that $v\not\in \supp w_1$
and \[\supp (\phi(\Gamma,T)(w_1)) \cap \lk_{\yt\Gamma}(v') = \varnothing.\]
Since $v^{\pm1} w_1 v^{\mp1}$ is reduced, there exists a vertex $x\in \supp w_1 \cap \lk_\Gamma(v)$.
Then \[p^{-1}(x)\cap \lk_{\yt\Gamma}(v') = \{x'\} \subseteq \lk_{\yt\Gamma}(v')\subseteq T_1\] for some vertex $x'$.
Since $x'\not\in\supp \phi(\Gamma,T)(w_1)$,
there exists an innermost cancellation of $x'$ in the $\phi(\Gamma,T)$--homomorphic word of $w_1$.
This implies that $w_1$ has a subword \[x^{\pm1} w_2 x^{\mp1}\] such that
\begin{equation*}
\supp (\phi(\Gamma,T)(w_2)) \cap \lk_{\yt\Gamma}(x')=\varnothing.
\tag{*}
\end{equation*}
The restriction of $\phi(\Gamma,T_1)$ to $G(\Gamma\setminus v)$ is injective by the assumption and Lemma~\ref{l:surv} (1). 
In particular, \[\phi(\Gamma,T_1)(x^{\pm1}w_2 x^{\mp1}) \ne \phi(\Gamma,T_1)(w_2).\]
So we can find vertices 
\[ x''\in p^{-1}(x)\cap T_1, \quad y\in \supp(\phi(\Gamma,T_1)(w_2))\cap \lk_{\yt\Gamma}(x'')\subseteq T_1.\]
Since $x',x''\in p^{-1}(x)\cap T_1$ there is a deck transformation $\sigma\in \Sigma$ such that $\sigma(x'')=x'$.
Note that
\[\sigma(y) \in \supp(\phi(\Gamma,\sigma(T_1))(w_2))\cap \lk_{\yt\Gamma}(x')
\subseteq
\supp(\phi(\Gamma,T)(w_2))\cap \lk_{\yt\Gamma}(x').\]
This contradicts (*).

It remains to show that $\phi(\Gamma,T)$ is injective. Suppose $w$ is a reduced word in $\ker\phi(\Gamma,T)\setminus 1$.
By the assumption, $v\in \supp w$ and hence,
the $\phi(\Gamma,T)$--homomorphic word of $w$ contains an innermost cancellation of $v'$.
This is a contradiction, for $\phi(\Gamma,T)$ is $v'$-surviving.
\end{proof}

Theorem~\ref{t:main} is an immediate consequence of the following theorem.

\begin{thm}\label{t:tree}
Let $\Gamma$ be a finite connected graph 
and $p\co \yt\Gamma\to \Gamma$ be a universal cover.
Then there exists a finite tree $T$ in $\yt\Gamma$ such that $\phi(\Gamma,T)$ is a quasi-isometric group embedding.
\end{thm}

\begin{proof}
We fix an arbitrary order on $V(\yt\Gamma)$ so that a $\phi(\Gamma,\Lambda)$--homomorphic word is well-defined for each finite subgraph $\Lambda$ of $\yt\Gamma$.
Let $F_0$ be a maximal tree in $\Gamma$ and $F$ be a lift of $F_0$ in $\yt\Gamma$,
so that $p(V(F))=V(\Gamma)$.
Choose a valence-one vertex $v$ in $F_0$ and let $v'$ be the lift of $v$ in $F$.
Note that $\Gamma\setminus v$ is connected
and each component of $p^{-1}(\Gamma\setminus v)$ is a universal cover of $\Gamma\setminus v$.
Inducting on the number of vertices in $\Gamma$,
we may assume that for some finite tree $T_0$ contained in $p^{-1}(\Gamma\setminus v)\le \yt\Gamma$,
the map $\phi(\Gamma\setminus v, T_0)$ is a quasi-isometric group embedding.
By Lemma~\ref{l:key}, there is a finite tree $T_1$ containing $v'$ in $\yt\Gamma$ such that $\phi(\Gamma,T_1)$ is $v'$-surviving and injective.
Then for each $u\in V(\Gam)$, the restriction of $\phi(\Gamma,T_1)$ to $G(\Gamma\setminus u)$ is injective.
By a repeated application of Lemmas~\ref{l:surv} and~\ref{l:key}, we can enlarge $T_1$ to another tree $T\le \yt\Gamma$ such that $\phi(\Gamma,T)$ is $F$-surviving. Lemma~\ref{l:surv-qi} completes the proof.
\end{proof}

\begin{rem}\label{r:ranks}
Let $m = | V(\Gamma) |$.
In Lemma~\ref{l:key}, with an additional hypothesis that $d(v',T_0)\le 1$ we see that 
\[ |V(T_1)| \le | V(T_0)| + |\st(v')| \le |V(T_0)|+m.\] So 
$|V(T)| \le | \Sigma| |V(T_1)| \le |V(T_1)|^3 \le (|V(T_0)|+m)^3$.
Now let us consider the proof of Theorem~\ref{t:tree}. We can inductively choose $T_0$ which contains $F_0$,
and deduce that 
\[
|V(T)| \le 
((\cdots( | V(T_0)|+m)^3 + m)^3\cdots+m)^3\]
where the iteration is $m$ times.
An easy recursive argument shows that 
\[|V(T)|\le2^{2^{(m-1)^2}}.\]
\end{rem}

\section{Applications}\label{s:app}
\subsection{Pure braid groups}
Crisp and Wiest proved that if $\Gamma$ is a planar graph then ${G(\Gamma)}$ 
admits a quasi-isometric group embedding into
a pure braid group~\cite{CW2007}.
We give an alternative account of this fact, modulo a general result about faithful homomorphisms from right-angled Artin groups to mapping class groups and their geometric behavior; see \cite{CLM2012}.

\begin{lem}[See \cite{CW2007}]\label{l:planar}
If $\Gamma$ is a planar graph, then for some $n$ we have that ${G(\Gamma)}$ admits a quasi-isometric group embedding into $P_n$.
\end{lem}

Let us denote by $\PMod(S)$ the pure mapping class group on the surface $S$.
We let $S_{g,n}^b$ denote the surface of genus $g$ with $n$ punctures and $b$ boundary components. We write $S_{0,n}$ for $S_{0,n}^0$.

\begin{proof}
Since we have isomorphisms~\cite[Theorem 8]{CLM2014}
\[P_{n-1}\cong \PMod(S_{0,n-1}^1)\cong \PMod(S_{0,n})\times \bZ\]
it suffices to embed ${G(\Gamma)}$ into $\PMod(S_{0,n})$.
Embed $\Gamma$ in the sphere and replace each vertex $v$ of $\Gamma$ with a small disk.  Deform these disks along the edges of $\Gamma$ to get a configuration of disks $\{D_v\}$ in the sphere which are in bijective correspondence with the vertices of $\Gamma$ and which intersect precisely when the corresponding vertices in $\Gamma$ are adjacent.
We require that the boundaries of an overlapping pair of disks intersect at two points, and also that there are no triple intersection of disks.

We now introduce $n$ punctures in the disks as follows: puncture $D_v$ in the interior of each intersection with another disk $D_{u}$.
Then, introduce three punctures in the interior of 
\[D_v\setminus \bigcup_{u\ne v}D_{u}\]  
for each $v$, and also three punctures in the exterior of $\cup_v D_v$.
It is now easy to produce a pseudo-Anosov mapping class $\psi_v$ whose support is the entirety of $D_v$.  Evidently the collection $\{\psi_v\}\subset \PMod(S_{0,n})$ is a collection of nontrivial mapping classes of the multiply-punctured sphere which commute if and only if the corresponding vertices are not adjacent in $\Gamma$, and furthermore no two of these mapping classes generate a cyclic subgroup of $\PMod(S_{0,n})$.  By \cite[Theorem 1.1]{CLM2012}, we see that sufficiently high powers of these mapping classes generate a quasi-isometrically embedded subgroup of $\PMod(S_{0,n})$ isomorphic to ${G(\Gamma)}$. 
\end{proof}

The proof of Corollary \ref{c:braid} part (1) is immediate from Theorem~\ref{t:main} and the above lemma.
We further claim if $m = |V(\Gamma)|$, then ${G(\Gamma)}$ embeds into $P_n$ for 
\[n = 2^{2^{m^2}}.\]
To see this, we first embed ${G(\Gamma)}$ into $G(T)$ for some tree satisfying $|V(T)| \le 2^{2^{(m-1)^2}}$ as in the previous section.
Note from the proof of Lemma~\ref{l:planar} that,
for $G(T)$ to embed into $P_n$ we have only to place three punctures at each vertex and on the exterior, and one puncture at each edge.
So it suffices to consider $n\le |E(T)|+3|V(T)|+3 = 4|V(T)|+2$.
It follows that for $m\ge2$,
\[n\le 4\cdot 2^{2^{(m-1)^2}} +2 \le 2^{2^{m^2}}.\]

\subsection{Area--preserving diffeomorphisms of the 2--sphere}\label{ss:sph}
We let $M=S^2$ and $\cG=\symps$, where $M$ is equipped with the round metric.
For $1\le p\le \infty$ and a vector field $X$ on $M$, we define
\[
\|X\|_p = 
\begin{cases} \left(\int_M |X|^p \; dx\right)^{1/p} &\mbox{if } p<\infty \\ 
\sup_{x\in M} |X(x)|& \mbox{if } p=\infty
\end{cases}
\]
For each path $\alpha\co I\to\cG$, we define the \emph{$L^p$--length} as
\[
\ell_p(\alpha) = \int_0^1 \|\partial \alpha/\partial t\|_p dt.\]
Then we have a non-degenerate right-invariant metric on $\cG$ by:
\[ d_p(\phi,\psi)=\inf \ell_p(\alpha)\]
where the infimum is taken over all paths $\alpha\subseteq \cG$ from $\alpha(0)=\phi$ to $\alpha(1)=\psi$. We let $\|\phi\|_p=d_p(\Id,\phi)$ so that $\|\phi\circ\psi\|_p\le\|\phi\|_p+\|\psi\|_p$; see~\cite[IV.7.A]{AK1998}, \cite[II.3.6]{KW2009}.

Let us fix a set of distinct points $P=\{m_1,m_2,\ldots,m_n\}$ in $M$ for some $n$.
Choose small disjoint closed disks $D_1,D_2,\ldots,D_n$ such that each $D_i$ is centered at $m_i$.
We let $\cP_n(M)$ be the subgroup of $\cG$ consisting of diffeomorphisms which restrict to the identity on $D_1\cup D_2\cup\cdots D_n$. Following~\cite{BS2013}, we let $X_n$ be the configuration space of $n$ distinct points in $M$ with the base point $m=(m_1,\ldots,m_n)$.
The \emph{$n$-strand pure braid group on $M$} is 
$P_n(M)=\pi_1(X_n,m)$. 
If $f\in\cG$ and $x=(x_1,x_2,\ldots,x_n)\in X_n$, we let $f(x)$ mean $(f(x_1),f(x_2),\ldots,f(x_n))$.

From~\cite{Smale1959}, we have homotopy equivalences 
\[
\cG\simeq \operatorname{Diff}^+(M)\simeq SO(3).\]
So the universal cover $p:\tilde\cG\to\cG$ is a two--to--one map.
Recall that $\tilde\cG$ consists of homotopy classes of paths in $\cG$.
Let $\sigma\co I\to SO(3)$ denote the rotational isotopy from $\Id$ to itself by a full rotation.
For each path $\alpha\co I\to \cG$ from $\Id$ to a map $\phi\in\mathcal{P}_n(M)$, we define
$\mu_n(\alpha)\in P_n(M)$ to be the braid represented by $\alpha(I)(m)$.
The map $\mu_n$ lifts to a map from $p^{-1}(\cP_n(M))$ to $P_n(M)$,
which is still denoted as $\mu_n$.
We have a natural commutative diagram
\[
\xymatrix{
\tilde\cG \ar[d]^{p}
& p^{-1}(\cP_n(M))\ar[l]_-{\tiny\mathrm{incl.}}
\ar[r]^{\mu_n}\ar[d]^{p}
& 
P_n(M)
\ar[d]^{h}
\\
\cG
& \cP_n(M)\ar[l]_-{\tiny\mathrm{incl.}}
\ar[r]
&
\PMod(S_{0,n})
}
\]
Note that $h$ is the \emph{point-pushing map}, so $h$ is surjective and $\ker h=\bZ_2$ by the \emph{belt trick}; 
see~\cite[p.251]{FM2012}. 
The group $\PMod(M\setminus P)=\PMod(S_{0,n})$ consists of mapping classes on $M$ fixing each $m_i$ pointwise.
The groups $P_n(M)$ and $\PMod(S_{0,n})$ are given with word metrics.
Suppose $\phi\in\cP_n(M)$. We have $p^{-1}(\phi)=\{[\alpha],[\sigma\cdot\alpha]\}$ 
for some path $\alpha\co I\to \cG$ connecting $\Id$ to $\phi$.
In the above diagram, the square on the right gives
\[
\xymatrix{
p^{-1}(\phi)=\{[\alpha],[\sigma\cdot\alpha]\}
\ar[r]\ar[d]^{p}
& 
\{\mu_n(\alpha),\mu_n(\sigma\cdot\alpha)\}
\ar[d]^{h}
\\
\{\phi\}
\ar[r]
&
\{[\phi]\}
}
\]
Since $h$ is a quasi-isometry,
there exists $C_0>0$ such that 
\[
\frac1C_0 \| [\phi]\| - C_0 \le \min(\|\mu_n(\alpha)\|,\|\mu_n(\sigma\cdot\alpha)\|).\]
Brandenbursky and Shelukhin proved that every finitely generated free abelian group quasi-isometrically embeds into $\symps$; see~\cite{BS2013}. Using an estimate given in~\cite{BS2013}, we prove the following lemma.

\begin{lem}\label{l:est}
For $p>2$, there exists $C=C(n,p)>0$ such that for each $\phi\in\cP_n(M)$,
the word-length $\|[\phi]\|$ in $\PMod(M\setminus P)$
is at most $C\|\phi\|_p+C$.
\end{lem}

\begin{proof}
Let $\phi\in\cP_n(M)$ and $\alpha\co I\to \cG$ be an isotopy from $\Id$ to $\phi$.
Put \[D_0=D_1\times\cdots D_n\subseteq X_n.\]
Let $x$ be in the interior of $D_0$.
We denote by $\gamma(x)$ the component-wise geodesic from $m$ to $x$,
and define a loop in $X_n$ as 
\[\ell(\alpha,x)=\gamma(x)\cdot \alpha(I)(x)\cdot{\gamma(x)^{-1}}.\]
Let $\gamma_s(x)$ be the restriction of the path $\gamma(x)$ to $[0,s]$.
From a braid isotopy
\[H(s) = \gamma_s(x) \cdot \alpha(I)(\gamma(x)(s))\cdot(\gamma_s(x))^{-1}\]
we see that
 $[\ell(\alpha,x)]=\mu_n(\alpha)$.
 For the point-pushing map $h$,
we have $h(\mu_n(\alpha))=[\phi]$. Hence,
\[\frac1{C_0}\|[\phi]\|-C_0\le \|\mu_n(\alpha)\|=\|[\ell(\alpha,x)]\|.\]
We see from~\cite[Lemmas 1 and 2]{BS2013} that for $p>2$
there exists $C_2>0$ independent of $\phi$ and $\alpha$ satisfying
\begin{equation*}
 \int_{D_0} \|[\ell(\alpha,x)]\| \; dx\le C_2 \ell_p(\alpha)+C_2.
\tag{\#}
\end{equation*}
Finally, we obtain
\[
\mathrm{vol}(D_0)\left(\frac1C_0\| [\phi]\| - C_0\right)
\le
\int_{D_0}\|[\ell(\alpha,x)]\| \;dx \le C_2\ell_p(\alpha)+C_2.\]
Taking the infimum of the right-hand side, we have a desired inequality.
\end{proof}
Let us complete the proof of Corollary~\ref{c:braid} part (3). 
By Theorem~\ref{t:main} and Lemma~\ref{l:planar}, every right-angled Artin group ${G(\Gamma)}$ admits a quasi-isometric group embedding $f_0\co{G(\Gamma)} \to \PMod(S_{0,n})$ for some $n$.
We can choose $\psi_v$ in the proof of Lemma~\ref{l:planar} to be area-preserving; see~\cite[Theorem 12]{CW2007} or~\cite[Lemma III.3.5]{AK1998}.
So $f_0$ factors through $f_1\co {G(\Gamma)}\to \cP_n(M)$ for some $f_1$.
Let $w\in {G(\Gamma)}$. 
By Lemma~\ref{l:est}, there exists $C,C'>0$ such that for every $w\in{G(\Gamma)}$, we have
\[
\frac1{C'}\|w\|-C'\le \|f_0(w)\|=\| [f_1(w)]\| \le C \| f_1(w)\|_p + C.\]
By the following Lemma, we see that $f_1$ is a desired quasi-isometric group embedding.
\begin{lem}\label{l:gp}
Suppose $G$ is a finitely generated torsion-free group with a word-metric $\|\cdot\|$,
and $H$ is a group equipped with a non-degenerate right-invariant metric $d$.
For $h\in H$, we set $\|h\| = d(1,h)$.
If there exists $C>0$ and a group homomorphism $\phi\co G\to H$ such that every $g\in G$ satisfies
\[
\| g\| \le C\| \phi(g)\| + C\]
then $\phi$ is a quasi-isometric group embedding.
\end{lem}
\begin{proof}
Suppose $S$ is a finite generating set of $G$ which is used to define the given word-metric.
Then for $g=s_1\cdots s_\ell$ where $s_i\in S$ and $\ell=\|g\|$,
we have $\| \phi(g)\| = \|\prod_i \phi(s_i)\| \le \sup_{s\in S} \|\phi(s)\|\cdot \|g\|$.
So it suffices to show that $\phi$ is injective. Suppose $\phi(g)=1$.
Then for every $n\in \bZ$, we have 
\[
\|g^n\| \le C\|\phi(g)^n\| + C = C.\]
This implies that $\{g^n\co n\in\bZ\}$ is a finite set, and hence, $g=1$.
\end{proof}

\subsection{Area--preserving diffeomorphisms of the 2--disk}\label{ss:disk}
In order to prove Corollary~\ref{c:braid} part (2), we can use the argument in Section~\ref{ss:sph} almost in verbatim by substituting $M=D^2$ and $\cG=\diff$. Here, the $L^p$--metric is defined in the same manner.
This case is even simpler, as $\tilde{\cG}=\cG$ and $P_n\cong\PMod(S^1_{0,n})$ by \cite{Smale1959} and \cite[Theorem 9.1]{FM2012}.
By H\"older inequality, we may assume $p=1$.
The counterpart to Equation (\#) for $p=1$ is given in~\cite[Equations (8,12,13 and 14)]{Brandenbursky2012}, and hence, Lemma~\ref{l:est} again holds in verbatim. 
We remark that  for $p=2$, Equation (\#) is given in~\cite[Lemma 4]{BG2001}.
We will omit the details.

Alternatively, one can combine Theorem~\ref{t:main} with the following theorem of Crisp and Wiest to prove Corollary~\ref{c:braid} part (2) for $p=2$.
A right-angled Artin group is of \emph{planar type} if the defining graph can be realized as the incidence graph of simple closed curves in the plane~\cite{CW2007}.
\begin{thm}[{\cite[Theorem 12]{CW2007}}]\label{t:cw}
Every right-angled Artin group of planar type admits a quasi-isometric group embedding into $(\diff,d_2)$.
\end{thm}
\subsection{Two--generated subgroups of right-angled Artin groups}
In \cite{LM2010}, Leininger and Margalit show the following structure theorem for two--generated subgroups of pure braid groups:

\begin{thm}[\cite{LM2010}]
Every two--generator subgroup of a pure braid group is either abelian or free.
\end{thm}

In \cite{Baudisch1981}, Baudisch established the same fact for two--generated subgroups of right-angled Artin groups.  Baudisch's proof uses somewhat involved combinatorial arguments.
Combining Leininger and Margalit's result with our main theorem, we obtain a more transparent proof of Baudisch's result:

\begin{proof}[Proof of Corollary \ref{c:baud}]
Each right-angled Artin group embeds into a pure braid group, in which every two--generated subgroup is either abelian or free.
\end{proof}

\subsection{Hyperbolic manifolds}
It is known that every word-hyperbolic Coxeter group is virtually special~\cite{HW2010}. 
Examples of discrete cocompact hyperbolic reflection groups are known to exist up to dimension eight by the work of Bugaenko~\cite{Bugaenko1990}, \cite{Bugaenko1992}. By Corollary~\ref{c:vsp}, there exist examples of closed hyperbolic $n$--manifolds whose fundamental groups admit quasi-isometric group embeddings into pure braid groups for $n\le 8$.

More concretely, the commutator group of a right-angled Coxeter group is special~\cite{Droms2003}, \cite{Davis2008}.
In particular, the commutator subgroup of the reflection group of the all-right 120--cell in $\mathbb{H}^4$ provides a specific example of a closed hyperbolic 4--manifold group which is special~\cite{CW2007}. 
Theorem~\ref{t:main} implies that this 4--manifold group admits a quasi-isometric group embedding into a pure braid group.

For an arbitrary $n$, the existence of closed hyperbolic $n$--manifold groups inside of right-angled Artin groups follows from the more recent work of Bergeron--Wise and Bergeron--Haglund--Wise~\cite{BW2012,BHW2011}.

\subsection{The isomorphism, conjugacy, and membership problem for finitely presented subgroups of braid groups}
The proofs of Corollaries \ref{c:isoprob} and \ref{c:membership} are straightforward from the work of Bridson, after combining it with Theorem \ref{t:main}.

\begin{thm}[\cite{Bridson2012}, Theorems 1.1 and 1.2]\label{t:bridson}
There exists a right-angled Artin group $A_1$ such that the isomorphism problem for finitely presented subgroups of $A_1$ is unsolvable.  Furthermore, there exists a right-angled Artin group $A_2$ and a finitely presented subgroup $H<A_2$ such that the conjugacy and membership problems for $H$ are unsolvable.
\end{thm}

\begin{proof}[Proofs of Corollaries \ref{c:isoprob} and \ref{c:membership}]
Let $A_1$ and $A_2$ be as in Theorem \ref{t:bridson}.  Then Theorem \ref{t:main} implies that $A_1$ and $A_2$ embed into $B_n$ for every sufficiently large $n$.
\end{proof}

\bibliographystyle{amsplain}
\bibliography{./ref}

\end{document}